\documentclass[11pt,letterpaper]{amsart}
\usepackage{graphicx} % Required for inserting images
\usepackage{amscd}
\usepackage{amssymb}
\usepackage{amsthm}
\usepackage{enumerate}
\usepackage[normalem]{ulem}
\usepackage{mathtools}
\usepackage[usenames,dvipsnames]{xcolor}
\usepackage{stmaryrd}
\usepackage[left=1in,top=1in,right=1in]{geometry}
\usepackage[T2A]{fontenc}
\usepackage[utf8]{inputenc}
\usepackage{mathrsfs}
\usepackage{hyperref}
\usepackage{comment}

\newtheorem{thm}{Theorem}

\newtheorem{prop}[thm]{Proposition}
\newtheorem{cor}[thm]{Corollary}

\newtheorem*{mthm}{Main Theorem}

\theoremstyle{definition}
\newtheorem{defn}[thm]{Definition}

\newtheorem{rem}[thm]{Remark}
\newtheorem*{question}{Question}

\newcommand{\Sym}{\textrm{Sym}}

\title{All actions of LERF groups on sets are sofic}
\author{David Gao}

\begin{document}

\begin{abstract}
In this paper, we shall prove that all actions of LERF groups on sets are sofic. As a corollary, we obtain that a large class of generalized wreath products are sofic.
\end{abstract}

\maketitle

\section{Introduction}

In [GKEP24], a notion of sofic actions on sets, heavily inspired by [HS18], was introduced and applications to proving soficity of generalized wreath products were presented. In it, two classes of groups were shown to have all their actions being sofic, namely amenable groups and free groups. In this paper, we shall identify a new large class of groups, containing free groups and surface groups, such that any action by any such group is sofic.

We recall the following definition of sofic actions on sets, as first appeared in [GKEP24]:

\begin{defn}{[GKEP24]}
Let $G$ be a countable discrete group, $X$ be a countable discrete set, $\alpha: G \curvearrowright X$ be an action, $A$ be a finite set, $\varphi: G \rightarrow \Sym(A)$ be a map (not necessarily a homomorphism):
\begin{enumerate}
    \item $\varphi$ is called \textit{unital} if $\varphi(1_G) = 1$;
    \item For a finite subset $F \subseteq G$ and $\epsilon > 0$, $\varphi$ is called $(F, \epsilon)$-\textit{multiplicative} if $d(\varphi(gh), \varphi(g)\varphi(h)) < \epsilon$ for all $g, h \in F$, where $d$ is the normalized Hamming distance on $\Sym(A)$;
    \item For finite subsets $F \subseteq G$, $E \subseteq X$, and $\epsilon > 0$, $\varphi$ is called an $(F, E, \epsilon)$-\textit{orbit approximation of} $\alpha$ if there exists a finite set $B$ and a subset $S \subseteq A$ s.t. $|S| > (1 - \epsilon)|A|$ and for each $s \in S$ there is an injective map $\pi_s: E \hookrightarrow B$ s.t. $\pi_{\varphi(g)s}(x) = \pi_s(\alpha(g^{-1})x)$ for all $s \in S$, $g \in F$, $x \in E$, whenever $\varphi(g)s \in S$ and $\alpha(g^{-1})x \in E$;
    \item $\alpha$ is called \textit{sofic} if for all finite subsets $F \subseteq G$, $E \subseteq X$, and $\epsilon > 0$, there exists a finite set $A$ and a map $\varphi: G \rightarrow \Sym(A)$ which is unital, $(F, \epsilon)$-multiplicative, and an $(F, E, \epsilon)$-orbit approximation of $\alpha$.
\end{enumerate}
\end{defn}

We recall the following definition, which is well-known in group theory:

\begin{defn}
A group $G$ is called \textit{locally extended residually finite (LERF)} if for any finitely generated subgroup $H \leq G$, there exists a decreasing sequence of finite index subgroups $H_i \leq G$ such that $H = \cap_i H_i$.
\end{defn}

Large classes of groups are known to be LERF. For example,
\begin{enumerate}
    \item By [Hall49], free groups are LERF.
    \item By [Sco78], surface groups are LERF.
    \item By [Wil08], limit groups are LERF. This generalizes both the cases of free groups and surface groups.
    \item Certain 3-manifold groups are LERF. For example, by Theorem 4.1 of [Sco78], fundamental groups of compact Seifert fiber spaces are LERF. By [Mal58], Sol manifold groups are LERF. By [Ago13] and [Wis11], hyperbolic 3-manifold groups are LERF. Though, not all 3-manifold groups are LERF. For example, mixed 3-manifold groups are not LERF. See [Sun19].
    \item A large class of right-angled Artin groups are LERF. Namely, all right-angled Artin groups associated to graphs without any square or path of length 3 as full subgraphs are LERF. See Theorem 2 of [MR08]. For a more general characterization of when Artin groups are LERF, see [EL24].
    \item By [Bur71], free products of LERF groups are LERF.
\end{enumerate}

The main theorem of this paper is the following:

\begin{mthm}
Let $G$ be a LERF group. Then all actions of $G$ on sets are sofic.
\end{mthm}

This vastly generalizes Theorem 2.19 of [GKEP24]. Furthermore, by applying Theorems 3.6, 3.7, 3.8, and 3.9 of [GKEP24], and observing that any LERF group is residually finite, so in particular sofic and hyperlinear, we will obtain the following corollary:

\begin{cor}
Let $G$ be a LERF group. Let $\alpha: G \curvearrowright X$ be an action on a set. Then,
\begin{enumerate}
    \item If $H$ is a sofic group, then the generalized wreath product $H \wr_\alpha G$ as well as the free generalized wreath product $H \wr^\ast_\alpha G$ is sofic.
    \item If $M$ is a Connes-embeddable tracial von Neumann algebra, then the generalized wreath product $M \wr_\alpha G$ as well as the free generalized wreath product $M \wr^\ast_\alpha G$ is Connes-embeddable.
\end{enumerate}
\end{cor}

For a definition of generalized wreath product and free generalized wreath product, see, for example, Definitions 3.1 and 3.2 of [GKEP24]. The technical result we will prove for the main theorem also allows us to prove the following:

\begin{thm}
Let $G$ be a residually finite group, $\alpha: G \curvearrowright X$ be an action on a set. If all stabilizers of $\alpha$ are centralizers of elements of $G$, then $\alpha$ is sofic.
\end{thm}

\textbf{Acknowledgements.} The author thanks Andreas Thom and anonymous referees for helpful comments and suggestions.

\section{Proof of Main Theorem}

To prove the main theorem, we shall prove the following technical result, which, in the present generality, is suggested by Andreas Thom. The author thanks him for his suggestion. Recall that the Chabouty topology on the collection of subgroups of $G$ is defined as the topology of pointwise convergence, i.e., subgroups $H_i \leq G$ converges to $H \leq G$ if the indicator functions $1_{H_i}$ converges to $1_H$ pointwise, as functions from $G$ to $\{0, 1\}$.

\begin{prop}\label{chabouty-closed}
The stabilizers of transitive sofic actions are closed in the Chabouty topology, i.e., if $H_i \leq G$ are s.t. the left multiplication action $\alpha_i: G \curvearrowright G/H_i$ is sofic for all $i$, and $H_i \to H$ in the Chabouty topology, then the left multiplication action $\alpha: G \curvearrowright G/H$ is sofic.
\end{prop}

\begin{proof} Let $F \subseteq G$ and $E \subseteq G/H$ be finite subsets. Let $\epsilon > 0$. Fix a lifting map $\sigma: G/H \rightarrow G$. Then we may consider the collection $\{\sigma(x)^{-1}\sigma(y): x \neq y \in E\}$. This is a finite set that does not intersect $H$, so for large $i$, $H_i$ does not intersect it. We may also consider the collection $\{\sigma(\alpha(g^{-1})x)^{-1}g^{-1}\sigma(x): x \in E, g \in F\}$. One easily observes that this is a finite subset of $H$, so for large $i$ it must be contained in $H_i$. Thus, we may fix an $i$ s.t.,
\begin{equation*}
\begin{split}
    \{\sigma(x)^{-1}\sigma(y): x \neq y \in E\} \cap H_i &= \varnothing\\
    \{\sigma(\alpha(g^{-1})x)^{-1}g^{-1}\sigma(x): x \in E, g \in F\} &\subseteq H_i
\end{split}
\end{equation*}

Now, let $q: G \rightarrow G/H_i$ be the natural quotient map. As the left multiplication action $\alpha_i: G \curvearrowright G/H_i$ is sofic, there exists $\varphi: G \rightarrow \Sym(A)$ that is a unital, $(F, \epsilon)$-multiplicative, $(F, q \circ \sigma(E), \epsilon)$-orbit approximation of $\alpha_i$. We claim it is also an $(F, E, \epsilon)$-orbit approximation of $\alpha$. Indeed, by definition there exists a finite set $B$, a subset $S \subseteq A$ with $|S| > (1 - \epsilon)|A|$, and for each $s \in S$, an injective map $\pi'_s: q \circ \sigma(E) \hookrightarrow B$ s.t. $\pi'_{\varphi(g)s}(x) = \pi'_s(\alpha_i(g^{-1})x)$ for all $s \in S$, $g \in F$, $x \in q \circ \sigma(E)$, whenever $\varphi(g)s \in S$ and $\alpha_i(g^{-1})x \in q \circ \sigma(E)$. Define $\pi_s(x) = \pi'_s(q \circ \sigma(x))$ for $s \in S$, $x \in E$. Then, if $s \in S$, $g \in F$, $x \in E$, $\varphi(g)s \in S$, and $\alpha(g^{-1})x \in E$, we have $q \circ \sigma(x) \in q \circ \sigma(E)$. Furthermore, $\alpha_i(g^{-1})q \circ \sigma(x) = q(g^{-1}\sigma(x)) = g^{-1}\sigma(x)H_i$. Since $\sigma(\alpha(g^{-1})x)^{-1}g^{-1}\sigma(x) \in H_i$ by the choice of $H_i$, we have $\alpha_i(g^{-1})q \circ \sigma(x) = \sigma(\alpha(g^{-1})x)H_i = q \circ \sigma(\alpha(g^{-1})x) \in q \circ \sigma(E)$. Hence,
\begin{equation*}
\begin{split}
    \pi_{\varphi(g)s}(x) &= \pi'_{\varphi(g)s}(q \circ \sigma(x))\\
    &= \pi'_s(q \circ \sigma(\alpha(g^{-1})x))\\
    &= \pi_s(\alpha(g^{-1})x)
\end{split}
\end{equation*}

Finally, we observe that, if $\pi_s(x) = \pi_s(y)$ for some $s \in S$ and $x, y \in E$, then by injectivity of $\pi'_s$ we have $\sigma(x)H_i = q \circ \sigma(x) = q \circ \sigma(y) = \sigma(y)H_i$, i.e., $\sigma(x)^{-1}\sigma(y) \in H_i$. By the choice of $H_i$, we have $x = y$, so $\pi_s$ is injective.
\end{proof}

\begin{cor}\label{LERF-cor}
If $G$ is LERF, then all actions of $G$ are sofic.
\end{cor}

\begin{proof} By Proposition 2.16 of [GKEP24], it suffices to consider transitive actions of $G$, i.e., it suffices to consider left multiplication actions $\alpha: G \curvearrowright G/H$ for subgroups $H$ of $G$. By Proposition \ref{chabouty-closed}, we may assume $H$ is finitely generated. As $G$ is LERF, there exists a decreasing sequence of finite index subgroups $H_i \leq G$ s.t. $H = \cap_i H_i$. It is easy to see that the left multiplication action of $G$ on $G/H_i$ is sofic for all $i$. (Take $A = \Sym(G/H_i)$, $\varphi: G \to \Sym(A) = \Sym(\Sym(G/H_i))$ be the map induced by the left multiplication action of $G$ on $G/H_i$, $S = A$, $B = G/H_i$, and $\pi_s(x) = s^{-1}x$.) Then applying Proposition \ref{chabouty-closed} again, we have $\alpha: G \curvearrowright G/H$ is sofic, whence the claim is proved.
\end{proof}

\begin{rem}\label{inc-union}
Applying item 4 of Proposition 2.15 of [GKEP24], we see that if $G$ is the union of an increasing sequence of LERF groups, then all its actions are sofic as well.
\end{rem}

We observe that, as applied to finite index $H_i \leq G$, the proof that $G \curvearrowright G/H_i$ is sofic contained in the proof of Corollary \ref{LERF-cor} shows that $\epsilon$ may be taken to be 0 in the definition of sofic actions and that $\varphi$ may be furthermore taken to be an actual homomorphism. Since $\epsilon = 0$ and $\varphi$ being a homomorphism are both preserved in the proof of Proposition \ref{chabouty-closed}, and the same holds in the proof of Proposition 2.16 of [GKEP24], we observe that $\epsilon = 0$ and $\varphi$ being a homomorphism may be assumed whenever $G$ is LERF, i.e.,

\begin{prop}
Let $G$ be a LERF group, $X$ be a countable discrete set, $\alpha: G \curvearrowright X$ be an action. Then for all finite subsets $F \subseteq G$, $E \subseteq X$, there exists a finite set $A$ and a group homomorphism $\varphi: G \rightarrow \Sym(A)$ which satisfies the condition that, for each $a \in A$ there is an injective map $\pi_a: E \hookrightarrow B$ s.t. $\pi_{\varphi(g)a}(x) = \pi_a(\alpha(g^{-1})x)$ for all $a \in A$, $g \in F$, $x \in E$, whenever $\alpha(g^{-1})x \in E$.
\end{prop}

One natural question to ask is whether the above characterizes LERF groups, at least in the finitely generated case, i.e.,

\begin{question}
If $G$ is a finitely generated group that satisfies the conclusion of the above proposition, is $G$ necessarily LERF?
\end{question}

This concludes the proof of the main theorem. In a similar vein, we also have the following generalization:

\begin{defn}
Let $H \leq G$ be an inclusion of groups. We say $H$ is \textit{strongly co-amenable in }$G$ if the normal core $\cap_{g \in G} (gHg^{-1})$ is co-amenable in $G$, i.e., if the quotient group $G/[\cap_{g \in G} (gHg^{-1})]$ is amenable. We say that a group $G$ is \textit{locally extended residually amenable (LERA)} if for any finitely generated subgroup $H \leq G$, there exists a decreasing sequence of strongly co-amenable subgroups $H_i \leq G$ s.t. $H = \cap_i H_i$.
\end{defn}

If $H \leq G$ is strongly co-amenable, then the left multiplication action $\alpha: G \curvearrowright G/H$ quotients through the amenable group $G/[\cap_{g \in G} (gHg^{-1})]$. Thus, $\alpha$ is sofic by item 1 of Proposition 2.15 and Theorem 2.17 of [GKEP24]. Therefore, by the same argument as Corollary \ref{LERF-cor}, we have,

\begin{cor}\label{LERA-cor}
If $G$ is LERA, then all actions of $G$ are sofic. Similar to Remark \ref{inc-union}, the same holds if $G$ is an increasing union of LERA groups.
\end{cor}

Since the normal core of a finite index subgroup is of finite index, LERF groups are LERA. Since quotients of amenable groups are amenable, amenable groups are also LERA. Thus, the above simultaneously generalizes the main theorem of this paper and Theorem 2.17 of [GKEP24]. LERA groups therefore encompass all known examples of finitely generated groups which have all their actions sofic. (For infinitely generated groups, there is the possibility of taking increasing unions.) This naturally yields the following question:

\begin{question}
If $G$ is a finitely generated group s.t. all its actions are sofic, is $G$ necessarily LERA?
\end{question}

However, we remark here that there are currently no known examples of non-sofic actions, unless one assumes the existence of non-sofic groups, in which case the left multiplication action of a non-sofic group on itself is non-sofic (see Theorem 2.12 of [GKEP24]).

We also observe that LERA groups are residually amenable, so in particular sofic and hyperlinear. Hence, by applying Theorems 3.6, 3.7, 3.8, and 3.9 of [GKEP24],

\begin{cor}
Let $G$ be a LERA group. Let $\alpha: G \curvearrowright X$ be an action on a set. Then,
\begin{enumerate}
    \item If $H$ is a sofic group, then the generalized wreath product $H \wr_\alpha G$ as well as the free generalized wreath product $H \wr^\ast_\alpha G$ is sofic.
    \item If $M$ is a Connes-embeddable tracial von Neumann algebra, then the generalized wreath product $M \wr_\alpha G$ as well as the free generalized wreath product $M \wr^\ast_\alpha G$ is Connes-embeddable.
\end{enumerate}
\end{cor}

Finally, we present the following additional application of Proposition \ref{chabouty-closed}:

\begin{prop}\label{left-right-of-res-fin}
Let $G$ be a residually finite group, then the action $\alpha: G \times G \curvearrowright G$, where the two copies of $G$ act via left multiplication and right multiplication, respectively, is sofic.
\end{prop}

\begin{proof} It suffices to show $G \times G$ is separable over stabilizers of $\alpha$, i.e., for any $g \in G$, there exists a decreasing sequence of finite index subgroups $H_i \leq G \times G$ s.t. $\mathrm{Stab}_\alpha(g) = \cap_i H_i$. Since $G$ is residually finite, we may take finite index normal subgroups $K_i \trianglelefteq G$ which decreases to $\{1_G\}$. $G \times G$ naturally acts on $G/K_i$ via left and right multiplication. Denote this action by $\alpha_i$ and set $H_i = \mathrm{Stab}_{\alpha_i}(gK_i) = \{(h, k) \in G \times G: g^{-1}hgk^{-1} \in K_i\}$. It is easy to see that $H_i$ satisfies the desired conditions.
\end{proof}

\begin{cor}\label{conj-of-res-fin}
Let $G$ be a residually finite group, then the conjugation action $\alpha: G \curvearrowright G$ is sofic. Furthermore, if $G$ is a residually finite group, $\alpha: G \curvearrowright X$ is an action on a set, and all stabilizers of $\alpha$ are centralizers of elements of $G$, then $\alpha$ is sofic.
\end{cor}

\begin{proof} The actions involved are all restrictions of the left and right multiplication action of $G \times G$ on $G$, so the result follows immediately from Proposition \ref{left-right-of-res-fin} and Propositions 2.15 and 2.16 of [GKEP24].
\end{proof}

\newpage

\begin{center}
    \textsc{References}
\end{center}

\medskip

\leftskip 0.25in
\parindent -0.25in
[Ago13] Agol, I. \textit{The virtual Haken conjecture}, with an appendix by Agol, I., Groves, D., and Manning, J. Documenta Mathematica. 18 (2013), pp. 1045-1087.

\smallskip

[Bur71] Burns, R.G. \textit{On finitely generated subgroups of free products}. Journal of the Australian Mathematical Society. 12 (1971), pp. 358-364.

\smallskip

[EL24] Emiliano de Almeida, K. and Lima, I. \textit{Subgroup separability of Artin groups II}. arXiv preprint, arXiv: 2403.05483, 2024.

\smallskip

[GKEP24] Gao, D., Kunnawalkam Elayavalli, S., and Patchell, G. \textit{Soficity for group actions on sets and applications}. arXiv preprint, arXiv: 2401.04945, 2024.

\smallskip

[Hall49] Hall, M. \textit{Coset representations in free groups}. Transactions of the American Mathematical Society. 67 (1949), pp. 421-432.

\smallskip

[HS18] Hayes, B. and Sale, A.W. \textit{Metric approximations of wreath products}. Annales de l'Institut Fourier (Grenoble). 68 (2018), pp. 423-455.

\smallskip

[Mal58] Mal'cev, A. \textit{On homomorphisms onto finite groups}. Ivanov. Gos. Ped. Inst. Ucen. Zap. 18 (1958), pp. 49-60.

\smallskip

[MR08] Metaftsis, V. and Raptis, E. \textit{On the profinite topology of right-angled Artin groups}. Journal of Algebra. 3 (2008), pp. 1174-1181.

\smallskip

[Sco78] Scott, P. \textit{Subgroups of surface groups are almost geometric}. Journal of the London Mathematical Society. 17 (1978), pp. 555-565.

\smallskip

[Sun19] Sun, H. \textit{NonLERFness of arithmetic hyperbolic manifold groups and mixed 3-manifold groups}. Duke Mathematical Journal. 4 (2019), pp. 655-696.

\smallskip

[Wil08] Wilton, H. \textit{Hall's theorem for limit groups}. Geometric and Functional Analysis. 18 (2008), pp. 271-303.

\smallskip

[Wis11] Wise, D. \textit{The structure of groups with a quasiconvex hierarchy}. preprint. 2011. retrieved from \url{https://www.math.mcgill.ca/wise/papers.html}.
\end{document}